\documentclass[a4paper, 12pt]{amsart}
\usepackage{graphicx}
\usepackage{amsmath,amssymb}

\newtheorem{thm}{Theorem}[section]

\newtheorem{lem}[thm]{Lemma}
\newtheorem{prop}[thm]{Proposition}

\theoremstyle{definition}

\theoremstyle{remark}
\newtheorem{rem}[thm]{Remark}
\theoremstyle{example}

\theoremstyle{conjecture}

\numberwithin{equation}{section}

\def\dlim{\displaystyle\lim}

\newcommand{\CC}{{\mathbb C}}

\newcommand{\calF}{{\mathcal F}}

\newcommand{\calO}{{\mathcal O}}

\begin{document}

\title[Differences of weighted composition operators]{Differences of weighted composition operators between the Fock spaces}

\author{Pham Trong Tien$^1$ \& Le Hai Khoi$^2$}%
\address{(Tien) Department of Mathematics, Mechanics and Information Technology, Hanoi University of Science, VNU, 334 Nguyen Trai, Hanoi, Vietnam}%
\email{phamtien@mail.ru, phamtien@vnu.edu.vn}

\address{(Khoi) Division of Mathematical Sciences, School of Physical and Mathematical Sciences, Nanyang Technological University (NTU),
637371 Singapore}%
\email{lhkhoi@ntu.edu.sg}

\thanks{$^1$ Supported in part by NAFOSTED grant No. 101.02-2017.313. \\
\indent $^2$ Supported in part by MOE's AcRF Tier 1 M4011724.110 (RG128/16)}



\date{\today}%



\begin{abstract}
We study some important topological properties such as boundedness, compactness and essential norm of differences of weighted composition operators between Fock spaces.
\end{abstract}




\maketitle

\section{Introduction}

Composition operators and weighted composition operators have been extensively investigated on various Banach spaces of holomorphic functions in the unit disc or the unit ball during the past several decades (see \cite{CM-95,S-93} and references therein for more information). 
Later, much progress was made in the study of such operators on Fock spaces of entire functions (see, for instance, \cite{CMS-03, TL-14, U-07}). 
There is a great number of topics on such types of operators: boundedness and compactness, topological structure, dynamical and ergodic properties, etc. In recent years, the study of differences of (weighted) composition operators acting on various Banach spaces of holomorphic functions on the unit disk or the unit ball has been received more attention. In details, many related results have been obtained on Hardy spaces \cite{SS-90}, Bergman spaces \cite{M-05}, Bloch spaces \cite{H-09, HO-07}, spaces of bounded holomorphic fucntions $H^{\infty}$ \cite{HI-05}, weighted Banach spaces with sup-norm \cite{M-08}, and even from Bloch spaces to spaces $H^{\infty}$\cite{HO-11,SZ-13}. Contrary to the context of the above-mentioned spaces, the similar question on Fock spaces is quite different and has been studied only for composition operators (more precisely, for linear sums of two composition operators) on Hilbert Fock spaces \cite{CIK-10}. 

The aim of this paper is to extend this problem for weighted composition operators between different Fock spaces.

For a number $p \in (0,\infty)$, the Fock space $\calF^p$ is defined as follows
$$
\calF^p: = \left\{ f \in \calO(\CC): \|f\|_{p}: = \left( \dfrac{p}{2\pi} \int_{\CC} |f(z)|^p e^{- \frac{p |z|^2}{2}}dA(z) \right)^{1/p} < \infty \right\},
$$
where $\calO(\CC)$ is the space of entire functions on $\CC$ with the usual compact open topology and $dA$ is the Lebesgue measure on $\CC$. 
Furthermore, the space $\calF^{\infty}$ consists of all entire functions $f \in \calO(\CC)$ for which 
$$
\|f\|_{\infty}: = \sup_{z \in \CC}|f(z)|e^{-\frac{|z|^2}{2}} < \infty.
$$
It is well known that $\calF^p$ with $ 1 \leq p \leq \infty$ is a Banach space. When $0 < p < 1$, $\calF^p$ is a complete metric space with the distance $d(f,g): = \|f-g\|_{p}^p$.

For each $w \in \CC$, we define the functions
$$
K_{w}(z): = e^{\overline{w} z} \text{ and } k_{w}(z): = e^{\overline{w} z - \frac{|w|^2}{2}}, z \in \CC.
$$
Then, $\|K_{w}\|_{p} = e^{\frac{|w|^2}{2}}$ and $\|k_{w}\|_{p} = 1$ for all $ 0 < p \leq \infty$ and $w \in \CC$; and $k_w$ converges to $0$ in the space $\calO(\CC)$ as $|w| \to \infty$.

We refer the reader to monograph \cite{KZ} for more details about Fock spaces. The following two auxiliary results will be used in the sequel.
\begin{lem}\label{Fp}
Let $p \in (0,\infty)$. For every $f \in \calF^p$ and $z \in \CC$,
$$
|f(z)| \leq e^{\frac{|z|^2}{2}} \|f\|_p.
$$
\end{lem}

\begin{lem}\label{lem-com}
Let $p, q \in (1,\infty)$. If the operator $T: \mathcal F^p \to \mathcal F^q$ is compact, then for every sequence $(w_n)_n$ in $\CC$ with $\dlim_{n \to \infty} |w_n| = \infty$, the sequence $(Tk_{w_n})_n$ converges to $0$ in $\mathcal F^q$.
\end{lem}

Note that the proofs of Lemmas \ref{Fp} and \ref{lem-com} were given in \cite[Corollary 2.8]{KZ} and \cite[Lemma 2.4]{TK-17}, respectively.

Let $\psi, \varphi$ be entire functions on $\CC$. The \textit{weighted composition operator} $W_{\psi, \varphi}$ is defined by $W_{\psi, \varphi}f = \psi (f \circ \varphi)$. When the function $\psi$ is identically $1$, the operator $W_{\psi, \varphi}$ reduces to the \textit{composition operator} $C_{\varphi}$.
As in \cite{TL-14, TK-17}, the following quantities play an important role in the present paper:
$$
m_z(\psi, \varphi): = |\psi(z)| e^{\frac{|\varphi(z)|^2 - |z|^2}{2}}, z \in \CC,
$$
and
$$
m(\psi,\varphi): = \sup_{z \in \CC} m_z(\psi, \varphi).
$$
We also put
$$
\rho(w_1,w_2):= \dfrac{|w_1 - w_2|^2}{2(2+|w_1 - w_2|^2)} \text{ for } w_1, w_2 \in \mathbb{C}.
$$

The topological properties such as boundedness, compactness and essential norm for a single weighted composition operator $W_{\psi, \varphi}$ acting between different Fock spaces were systematically studied in \cite[Section~3]{TK-17}. Hereby, for the reader's convenience we summarize the main results of this section which will be used in this paper.

\begin{thm} \label{thm-swco}
The following statements are true.
\begin{itemize}
\item[(i)]
For every $p, q \in (0,\infty)$, if the operator $W_{\psi, \varphi}: \calF^p \to \calF^q$ is bounded, then $\psi \in \calF^q$ and $m(\psi, \varphi) < \infty$. In this case, $\varphi(z) = az + b$ with $|a| \leq 1$. Moreover, if $a = 0$, i.e., $\varphi(z) = b$ and $\psi \in \calF^q$, then $W_{\psi, \varphi}: \calF^p \to \calF^q$ is compact.
\item[(ii)] Let $\psi$ be a nonzero function in $\calF^q$ and $\varphi(z) = az + b$ with $0 < |a| \leq 1$.

- In the case $ 0 < p \leq q < \infty$, the operator $W_{\psi, \varphi}: \calF^p \to \calF^q$ is bounded (or compact) if and only if $m(\psi, \varphi) < \infty$ (or, respectively, $\dlim_{|z| \to \infty} m_z(\psi, \varphi) = 0$). 

- In the case $0 < q < p < \infty$, the operator $W_{\psi, \varphi}: \calF^p \to \calF^q$ is bounded (or, equivalently, compact) if and only if $m_z(\psi, \varphi) \in L^{\frac{pq}{p-q}}(\mathbb C, dA)$. In particular, the operator $C_{\varphi}: \calF^p \to \calF^q$ is bounded (or, equivalently, compact) if and only if $|a| < 1$.
\end{itemize}
\end{thm}

In the present paper we study some topological properties for differences of weighted composition operators. Our main results are stated in the following theorems.

\begin{thm}\label{thm-m1}
Let $ 0 < p \leq q < \infty$ and $\psi_1, \psi_2, \varphi_1, \varphi_2$ be entire functions with $\psi_1 \not \equiv 0, \psi_2 \not \equiv 0$ and $\varphi_1 \not \equiv \varphi_2$. The following statements are true:
\begin{itemize}
\item[(a)] $W_{\psi_1, \varphi_1} - W_{\psi_2, \varphi_2}: \calF^p \to \calF^q$ is bounded if and only if both $W_{\psi_1, \varphi_1}$ and $W_{\psi_2, \varphi_2}: \calF^p \to \calF^q$ are bounded.
\item[(b)] $W_{\psi_1, \varphi_1} - W_{\psi_2, \varphi_2}: \calF^p \to \calF^q$ is compact if and only if both $W_{\psi_1, \varphi_1}$ and $W_{\psi_2, \varphi_2}: \calF^p \to \calF^q$ are compact.
\end{itemize}
\end{thm}

The case $0 < q < p < \infty$ is more complicated and we get the result only for a linear combination $c_1 C_{\varphi_1} + c_2 C_{\varphi_2}$. A similar problem for $W_{\psi_1, \varphi_1} - W_{\psi_2, \varphi_2}$ is till now unsolved.

\begin{thm}\label{thm-m2}
Let $0 < q < p < \infty$ and $\varphi_1, \varphi_2$ be entire functions with $\varphi_1 \not \equiv \varphi_2$
and $c_1, c_2$ nonzero constants. The following conditions are equivalent.
\begin{itemize}
\item[(i)] $c_1 C_{\varphi_1} + c_2 C_{\varphi_2}: \calF^p \to \calF^q$ is bounded.
\item[(ii)] $c_1 C_{\varphi_1} + c_2 C_{\varphi_2}: \calF^p \to \calF^q$ is compact.
\item[(iii)] Both $C_{\varphi_1}$ and $C_{\varphi_2}: \calF^p \to \calF^q$ are compact.
\end{itemize}
\end{thm}

In what follows, for a bounded linear operator $L$ from $\calF^p$ into $\calF^q$, the essential norm of $L$, denoted by $\|L\|_e$, is defined as  
$$
\|L\|_e: =\inf\{\|L-K\|: K \text{ is compact from } \calF^p \text{ into } \calF^q\}.
$$
Now we give estimates for essential norm of a difference $W_{\psi_1, \varphi_1} - W_{\psi_2, \varphi_2}$. By Theorem \ref{thm-m2}, we are only interested in the case $0 < p \leq q < \infty$. In this case, by Theorem \ref{thm-m1}, both $W_{\psi_1, \varphi_1}$ and $W_{\psi_2, \varphi_2}$ are bounded. Then, by Theorem \ref{thm-swco}, $\varphi_1(z) = a_1 z + b_1$ and $\varphi_2(z) = a_2 z + b_2$ with $|a_1| \leq 1$ and $|a_2| \leq 1$.

Obviously, if $a_1 = 0$ (or $a_2 = 0$), then, by Theorem \ref{thm-swco}, $W_{\psi_1, \varphi_1}$ (or, respectively, $W_{\psi_2, \varphi_2}$) is compact from $\calF^p$ into $\calF^q$, and hence, $\|W_{\psi_1, \varphi_1} - W_{\psi_2, \varphi_2}\|_e = \|W_{\psi_2, \varphi_2}\|_e$ (or, respectively, $\|W_{\psi_1, \varphi_1} - W_{\psi_2, \varphi_2}\|_e = \|W_{\psi_1, \varphi_1}\|_e$), estimates for which are given in \cite[Theorem~3.7]{TK-17}.

In this view, we only study the case when $a_1 \neq 0$ and $a_2 \neq 0$.

\begin{thm}\label{thm-ess}
Let $1 < p \leq q < \infty$ and $\psi_1, \psi_2, \varphi_1, \varphi_2$ be entire functions with $\psi_1 \not \equiv 0, \psi_2 \not \equiv 0$ and $\varphi_1 \not \equiv \varphi_2$. Suppose that $W_{\psi_1,\varphi_1} - W_{\psi_2,\varphi_2}: \calF^p \to \calF^q$ is bounded, and hence, $\varphi_1(z) = a_1 z + b_1, \varphi_2(z) = a_2 z + b_2$ with $0 < |a_1|, |a_2| \leq 1$. Then
\begin{align*}
& \alpha_{\varphi_1, \varphi_2} \limsup_{|z|\to\infty}  \left( m_z(\psi_1,\varphi_1) + m_z(\psi_2,\varphi_2) \right) \leq \|W_{\psi_1,\varphi_1} - W_{\psi_2,\varphi_2}\|_e \\ 
& \leq 2\left( \dfrac{q}{p|a_1|^2}\right)^{\frac{1}{q}} \limsup_{|z| \to \infty} m_z(\psi_1, \varphi_1) + 2\left( \dfrac{q}{p|a_2|^2}\right)^{\frac{1}{q}} \limsup_{|z| \to \infty} m_z(\psi_2, \varphi_2), 
\end{align*}
where
$$
\alpha_{\varphi_1, \varphi_2}: = 
\begin{cases}
\frac{1}{2}, \qquad \ \ \text{  if } a_1 \neq a_2 \\
\rho(b_1, b_2), \text{ if } a_1 = a_2.
\end{cases}
$$
\end{thm}

Note that in the settings of Theorem \ref{thm-ess}, the part (b) of Theorem \ref{thm-m1} is an immediate consequence of Theorem \ref{thm-ess}. Indeed, if $W_{\psi_1,\varphi_1} - W_{\psi_2,\varphi_2}$ is compact from $\calF^p$ into $\calF^q$, i.e. $\|W_{\psi_1,\varphi_1} - W_{\psi_2,\varphi_2}\|_e = 0$, then by the lower estimate of $\|W_{\psi_1,\varphi_1} - W_{\psi_2,\varphi_2}\|_e$,
$$ 
\limsup_{|z|\to\infty}  \left( m_z(\psi_1,\varphi_1) + m_z(\psi_2,\varphi_2) \right)= 0
$$
From this and Theorem \ref{thm-swco}, it follows that both $W_{\psi_1, \varphi_1}$ and $W_{\psi_2, \varphi_2}$ are compact. 

\section{Proofs of the main theorems}

We start with the following lower estimate of the norm of linear sums of two functions $K_{w_1}$ and $K_{w_2}$.

\begin{lem}\label{lem-est}
For every $\alpha_1, \alpha_2, w_1, w_2 \in \CC$, the following inequality holds
\begin{equation} \label{eq-est}
\|\alpha_1  K_{w_1} + \alpha_2  K_{w_2}\|_{\infty} \geq \rho(w_1,w_2) \left(|\alpha_1| \|K_{w_1}\|_{\infty} + |\alpha_2| \|K_{w_2}\|_{\infty}\right).
\end{equation}
\end{lem}

\begin{proof}
By the definition
$$
\|\alpha_1  K_{w_1} + \alpha_2  K_{w_2}\|_{\infty} \geq |\alpha_1  K_{w_1}(z) + \alpha_2  K_{w_2}(z)|e^{-\frac{|z|^2}{2}}, \ \forall z \in \CC.
$$
In particular, with $z = w_1$ the last inequality gives
\begin{align*}
\|\alpha_1  K_{w_1} + \alpha_2  K_{w_2}\|_{\infty} &\geq \left| \alpha_1  K_{w_1}(w_1) + \alpha_2  K_{w_2}(w_1) \right| e^{-\frac{|w_1|^2}{2}} \\
& = \left| \alpha_1  e^{\frac{|w_1|^2}{2}} + \alpha_2  e^{\overline{w_2}w_1 - \frac{|w_1|^2}{2}} \right| \\
& \geq |\alpha_1|  e^{\frac{|w_1|^2}{2}} - |\alpha_2|  \left| e^{\overline{w_2}w_1 - \frac{|w_1|^2}{2}} \right| \\
& = |\alpha_1|  e^{\frac{|w_1|^2}{2}} - |\alpha_2|  e^{\frac{|w_2|^2}{2} - \frac{|w_1-w_2|^2}{2}} \\
& = |\alpha_1| \|K_{w_1}\|_{\infty} - |\alpha_2| \|K_{w_2}\|_{\infty} e^{ - \frac{|w_1-w_2|^2}{2}}.
\end{align*}
Similarly with $z = w_2$ we have
$$
\|\alpha_1  K_{w_1} + \alpha_2  K_{w_2}\|_{\infty} \geq |\alpha_2| \|K_{w_2}\|_{\infty} - |\alpha_1| \|K_{w_1}\|_{\infty} e^{ - \frac{|w_2 - w_1|^2}{2}}.
$$
Combining these inequalities we get
$$
2\|\alpha_1  K_{w_1} + \alpha_2  K_{w_2}\|_{\infty} \geq \left(|\alpha_1| \|K_{w_1}\|_{\infty} + |\alpha_2| \|K_{w_2}\|_{\infty}\right) \left( 1 - e^{-\frac{|w_1 - w_2|^2}{2}} \right).
$$
From this and the fact that $ 1 - e^{-x} \geq \frac{x}{1+x}$ for all $x \geq 0$, the desired inequality follows.
\end{proof}

Next we give a necessary condition for the boundedness of a difference
$W_{\psi_1,\varphi_1} - W_{\psi_2,\varphi_2}$, which plays an essential role in this paper.

\begin{prop}\label{prop-m}
Let $p, q \in (0, \infty)$. If the difference $L: = W_{\psi_1,\varphi_1} - W_{\psi_2,\varphi_2}$ is bounded from $\calF^p$ into $\calF^q$, then $m(\psi_j, \varphi_j) < \infty$. In this case, $\varphi_j(z) = a_j z + b_j$  with $|a_j| \leq 1$ and $j=1, 2$.
\end{prop}
\begin{proof}
By Lemma \ref{Fp} and $\|k_w\|_p = 1$ for every $w \in \CC$, we see that for all $w, z \in \CC$,
\begin{eqnarray*}
\|L\| &\geq& \|W_{\psi_1, \varphi_1}k_w - W_{\psi_2, \varphi_2}k_w\|_{q} \\
&\geq& \left|\psi_1(z) e^{\overline{w}\varphi_1(z)-\frac{|w|^2}{2}} - \psi_2(z) e^{\overline{w}\varphi_2(z)-\frac{|w|^2}{2}} \right| e^{-\frac{|z|^2}{2}}\\
&=& \left|\overline{\psi_1}(z) e^{w\overline{\varphi_1}(z)} - \overline{\psi_2}(z) e^{w\overline{\varphi_2}(z)} \right| e^{-\frac{|z|^2+|w|^2}{2}}\\
&=& \left|\overline{\psi_1}(z) K_{\overline{\varphi_1}(z)}(w) - \overline{\psi_2}(z) K_{\overline{\varphi_2}(z)}(w) \right| e^{-\frac{|z|^2+|w|^2}{2}}.
\end{eqnarray*}
Taking supremum over $w \in \mathbb{C}$, we get
\begin{equation*}
\|L\| \geq e^{-\frac{|z|^2}{2}} \|\overline{\psi_1}(z) K_{\overline{\varphi_1}(z)} - \overline{\psi_2}(z) K_{\overline{\varphi_2}(z)} \|_{\infty}, \forall z \in \CC.
\end{equation*}
Applying Lemma \ref{lem-est} to 
$$
\alpha_1 = \overline{\psi_1}(z), \alpha_2 = - \overline{\psi_2}(z), w_1 = \overline{\varphi_1}(z), w_2 = \overline{\varphi_2}(z),
$$
we have 
\begin{eqnarray*}
\|L\| &\geq&  e^{-\frac{|z|^2}{2}} \rho(\overline{\varphi_1}(z), \overline{\varphi_2}(z))  \left( |\overline{\psi_1}(z)| \|K_{\overline{\varphi_1}(z)}\|_{\infty} + |\overline{\psi_2}(z)| \|K_{\overline{\varphi_2}(z)}\|_{\infty} \right) \\ 
& = & e^{-\frac{|z|^2}{2}}  \rho(\varphi_1(z), \varphi_2(z)) \left( |\psi_1(z)| e^{\frac{|\varphi_1(z)|^2}{2}}  + |\psi_2(z)| e^{\frac{|\varphi_2(z)|^2}{2}} \right)  \\ 
& = & \frac{|\varphi_1(z) - \varphi_2(z)|^2}{2(2 + |\varphi_1(z) - \varphi_2(z)|^2)} \left( m_z(\psi_1, \varphi_1) + m_z(\psi_2, \varphi_2) \right), \forall z \in \CC.
\end{eqnarray*}
This implies that, for all $z \in \mathbb C$,
\begin{equation} \label{eq-1} 
\|L\| \geq \frac{|\varphi_1(z) - \varphi_2(z)|^2}{2(2 + |\varphi_1(z) - \varphi_2(z)|^2)} m_z(\psi_j, \varphi_j) \text{ with } j = 1, 2.
\end{equation}
Taking logarithm both sides of this inequality, we get
\begin{align*}
\log|\psi_1(z)| & + 2\log|\varphi_1(z) - \varphi_2(z)| + \frac{|\varphi_1(z)|^2 - |z|^2}{2}   \\
 & \leq \log (2 + |\varphi_1(z) - \varphi_2(z)|^2) + \log 2\|L\|,
\end{align*}
and
\begin{align*}
\log|\psi_2(z)| & + 2\log|\varphi_1(z) - \varphi_2(z)| + \frac{|\varphi_2(z)|^2 - |z|^2}{2}  \\
 & \leq \log (2 + |\varphi_1(z) - \varphi_2(z)|^2) + \log 2\|L\|,
\end{align*}
for all $z \in \CC$. Then
\begin{align*}
\log|\psi_1(z)\psi_2(z)| &+ 4\log|\varphi_1(z) - \varphi_2(z)| + \frac{|\varphi_1(z)|^2 + |\varphi_2(z)|^2}{2} - |z|^2\\
& \leq 2\log (2 + |\varphi_1(z) - \varphi_2(z)|^2) + 2\log 2\|L\|, \ \forall z \in \CC.
\end{align*}
Since $\log (2 + x) \leq \frac{x}{16} + 2$ for all $x \geq 0$, 
\begin{align*}
\log|\psi_1(z)\psi_2(z)| &+ 4\log|\varphi_1(z) - \varphi_2(z)| + \frac{|\varphi_1(z) - \varphi_2(z)|^2}{4} - |z|^2 \\
& \leq \frac{|\varphi_1(z) - \varphi_2(z)|^2}{8} + 4 + 2 \log 2 \|L\|, \ \ \forall z \in \CC.
\end{align*}
Thus, for all $z \in \CC$,
\begin{align*}
\log|\psi_1(z)\psi_2(z)(\varphi_1(z) - \varphi_2(z))^4| +  \frac{|\varphi_1(z) - \varphi_2(z)|^2}{8} & - |z|^2 \\
& \leq 4 + 2 \log 2 \|L\|,
\end{align*}
i.e.,
\begin{align*}
\log|\psi_1(z)\psi_2(z)(\varphi_1(z) - \varphi_2(z))^4|^2 + \frac{|\varphi_1(z) - \varphi_2(z)|^2}{4} & - |\sqrt{2}z|^2 \\
& \leq 8 + 4 \log 2 \|L\|.
\end{align*}

Using this and applying \cite[Proposition 2.1]{TL-14} to entire functions 
$$
f(\zeta): = \psi_1(\zeta /\sqrt{2})\psi_2(\zeta /\sqrt{2}) \left(\varphi_1(\zeta /\sqrt{2}) - \varphi_2(\zeta /\sqrt{2})\right)^4
$$
and
$$
\varphi(\zeta): = \frac{\varphi_1(\zeta /\sqrt{2}) - \varphi_2(\zeta /\sqrt{2})}{2},
$$
we can conclude that $\varphi_1(z) - \varphi_2(z) = az +b$ for some $a,b \in \CC$. Then
$$
\lim_{|z| \to \infty} \frac{|\varphi_1(z) - \varphi_2(z)|^2}{2 + |\varphi_1(z) - \varphi_2(z)|^2} = 
\begin{cases}
1, \text{ if } a \neq 0, \\
\frac{|b|^2}{2 + |b|^2} > 0, \text{ if } a = 0. 
\end{cases}
$$
This and \eqref{eq-1} imply that $m(\psi_1, \varphi_1) < \infty$ and $m(\psi_2, \varphi_2) < \infty$. Then, by \cite[Proposition~2.1]{TL-14}, $\varphi_1(z) = a_1 z + b_1$, $\varphi_2(z) = a_2 z + b_2$ with $|a_1| \leq 1$ and $|a_2| \leq 1$.
\end{proof}

We are now ready to prove Theorems \ref{thm-m1} and \ref{thm-m2}.

\medskip

\textbf{Proof of Theorem \ref{thm-m1}}. The sufficient part of statements (a) and (b) is obvious. Moreover, the necessary part of (b) is an immediate consequence of \cite[Theorem~4.1]{TK-17} and (a). Then it is enough to prove the necessity of (a). 
Suppose that the difference $L:= W_{\psi_1, \varphi_1} - W_{\psi_2, \varphi_2}$ is bounded from $\calF^p$ into $\calF^q$. Then, by Proposition \ref{prop-m}, $m(\psi_j, \varphi_j) < \infty$ and $\varphi_j(z) = a_j z + b_j$ with $|a_j| \leq 1$ and $j = 1, 2$.
We consider the following two cases.

\textbf{Case 1.} Either of $a_1$ or $a_2$ is nonzero. 
If say $a_1 \neq 0$ (it is similar for the case $a_2 \neq 0$), then $\psi_1 \in \calF^q$. Indeed, for some $C>0$ and $\alpha \in (0, |a_1|)$, 
$$
|\psi_1(z)| \leq m(\psi_1, \varphi_1) e^{\frac{|z|^2 - |a_1 z + b_1|^2}{2}} \leq C e^{\frac{(1-\alpha^2)|z|^2}{2}}, \ \forall z \in \CC.
$$
It implies that $\psi_1 \in \calF^q$. 
By Theorem \ref{thm-swco}, $W_{\psi_1, \varphi_1}$ is bounded from $\calF^p$ into $\calF^q$, and hence, so is $W_{\psi_2, \varphi_2}$.

\textbf{Case 2.} Now suppose $a_1 = a_2 = 0$, i.e., $\varphi_1(z) \equiv b_1$ and $\varphi_2(z) \equiv b_2$ with $b_1 \neq b_2$. Then
$$
L(e_0) = \psi_1 - \psi_2 \in \calF^q \text{ and } L(e_1) = b_1 \psi_1 - b_2 \psi_2 \in \calF^q,
$$
where $e_0: =z^0$ and $e_1: = z^1$. Thus,
$$
\psi_1 = \frac{b_2 L(e_0) - L(e_1)}{b_2 - b_1} \in \calF^q  \text{ and }
\psi_2 = \frac{b_1 L(e_0) - L(e_1)}{b_2 - b_1} \in \calF^q.
$$
It remains to use Theorem \ref{thm-swco} to conclude that $W_{\psi_1, \varphi_1}$ and  $W_{\psi_2, \varphi_2}$ are bounded from $\calF^p$ into $\calF^q$.
\begin{flushright}
$\square$
\end{flushright}

\medskip

\textbf{Proof of Theorem \ref{thm-m2}}. Obviously (iii) $\Longrightarrow$ (ii) $\Longrightarrow$ (i). We prove that (i) $\Longrightarrow$ (iii).

Suppose the linear combination $L:= c_1 C_{\varphi_1} + c_2 C_{\varphi_2}$ is bounded from $\calF^p$ into $\calF^q$.
By Proposition \ref{prop-m}, $m(c_j, \varphi_j) < \infty$ and, in this case, $\varphi_j(z) = a_j z + b_j$ with $|a_j| \leq 1$ and $b_j = 0$ whenever $|a_j|=1$ with $j = 1, 2$.

We show by contradiction that $|a_1| < 1$ and $|a_2| < 1$, and then, by Theorem \ref{thm-swco}, both $C_{\varphi_1}$ and $C_{\varphi_2}: \calF^p \to \calF^q$ are compact.

Indeed, assume that, for instance, $|a_1| = 1$. Then, by Theorem \ref{thm-swco}, $C_{\varphi_1}: \calF^p \to \calF^q$ is unbounded. Then, $C_{\varphi_2}: \calF^p \to \calF^q$ is unbounded, too. By Theorem \ref{thm-swco} again, $|a_2|$ must be $1$. Thus, in this case $\varphi_j(z) = a_j z$ with $|a_j| = 1, j = 1, 2$ and $a_1 \neq a_2$.

Since $L: \calF^p \to \calF^q$ is bounded, for every $n \in \mathbb N$,
\begin{align*}
\|L\| \|z^n\|_p \geq \|Lz^n\|_q & = \left( \dfrac{q}{2\pi} \int_{\CC} |c_1 a_1^n z^n + c_2 a_2^n z^n |^q e^{-\frac{q|z|^2}{2}}dA(z) \right)^{\frac{1}{q}} \\
& = |c_1 a_1^n + c_2 a_2^n| \|z^n\|_q.
\end{align*}
On the other hand, 
$$
\|z^n\|_p = \left( \dfrac{2}{p} \right)^{\frac{n}{2}} \Gamma^{\frac{1}{p}}\left( \frac{np}{2} + 1 \right) \sim \left( \dfrac{n}{e} \right)^{\frac{n}{2}} (\pi p n)^{\frac{1}{2p}} \text{ as } n \to \infty,
$$
and, similarly, 
$$
\|z^n\|_q = \left( \dfrac{2}{q} \right)^{\frac{n}{2}} \Gamma^{\frac{1}{q}}\left( \frac{nq}{2} + 1 \right) \sim \left( \dfrac{n}{e} \right)^{\frac{n}{2}} (\pi q n)^{\frac{1}{2q}} \text{ as } n \to \infty.
$$
Therefore, $c_1 a_1^n + c_2 a_2^n \to 0$ as $n \to \infty$. We show that it is a contradiction. There are two cases.

\textbf{Case 1.} If $|c_1| \neq |c_2|$, then
$$
0 < \left| |c_1| - |c_2| \right| \leq |c_1 a_1^n + c_2 a_2^n| \to 0 \text{ as } n \to \infty,
$$
which is impossible. 

\textbf{Case 2.} Suppose now $|c_1| = |c_2|$ and put
$$
c_j = r e^{i\omega_j} \text{ and } a_j =  e^{i\theta_j} \text{ with } \omega_j, \theta_j \in [0,2\pi), j = 1, 2 \text{ and } \theta_1 \neq \theta_2.
$$
Then
\begin{align*}
\left| c_1 a_1^n + c_2 a_2^n \right| &= r \left| e^{i (\omega_1 + n \theta_1)} + e^{i (\omega_2 + n \theta_2)} \right|\\
 &= r  \left| 1 + e^{i ((\omega_2 - \omega_1) + n (\theta_2 - \theta_1))} \right| \\
 &= 2 r \left| \cos \frac{(\omega_2 - \omega_1) + n (\theta_2 - \theta_1)}{2} \right|. 
\end{align*}
Consequently,
$$
\cos \frac{(\omega_2 - \omega_1) + n (\theta_2 - \theta_1)}{2}  \to 0 \text{ as } n \to \infty,
$$
which is a contradiction.
\begin{flushright}
$\square$
\end{flushright}

\medskip

Finally, we prove Theorem \ref{thm-ess}.

\textbf{Proof of Theorem \ref{thm-ess}.}
Obviously, 
$$
\|L\|_e \leq \|W_{\psi_1, \varphi_1}\|_e + \|W_{\psi_2, \varphi_2}\|_e.
$$
Then the upper estimate of $\|L\|_e$ directly follows from \cite[Theorem~3.7]{TK-17}.

For the lower estimate, let $T$ be a compact operator from $\calF^p$ into $\mathcal F^q$. By Lemma \ref{Fp} and  $\|k_w\|_p =1$, for every $w \in \CC$, we see that for every $z \in \CC$,
\begin{align*}
\|L-T\| &\geq \|W_{\psi_1, \varphi_1}k_w - W_{\psi_2, \varphi_2} k_w - Tk_w\|_q \\
& \geq \|W_{\psi_1, \varphi_1}k_w - W_{\psi_2, \varphi_2} k_w\|_q - \|Tk_w\|_q \\
& \geq \left| \psi_1(z) e^{\overline{w}\varphi_1(z) -\frac{|w|^2}{2} } - \psi_2(z) e^{\overline{w}\varphi_2(z) - \frac{|w|^2}{2}} \right| e^{-\frac{|z|^2}{2}} - \|Tk_w\|_q.
\end{align*}
In particular, with $w = \varphi_1(z)$ the last inequality gives
\begin{align*}
\|L-T\| & \geq \left| \psi_1(z) e^{\frac{|\varphi_1(z)|^2 - |z|^2}{2}} - \psi_2(z) e^{\overline{\varphi_1(z)}\varphi_2(z)-\frac{|z|^2+|\varphi_1(z)|^2}{2}} \right| - \|Tk_{\varphi_1(z)}\|_q\\ 
& \geq \left| m_z(\psi_1, \varphi_1) - m_z(\psi_2,\varphi_2) e^{-\frac{|\varphi_1(z)-\varphi_2(z)|^2}{2}} \right| - \|Tk_{\varphi_1(z)}\|_q, \forall z \in \CC.
\end{align*}
Similarly, with $w = \varphi_2(z)$ we get
\begin{align*}
\|L-T\| & \geq \left| m_z(\psi_2, \varphi_2) - m_z(\psi_1,\varphi_1) e^{-\frac{|\varphi_2(z)-\varphi_1(z)|^2}{2}} \right| - \|Tk_{\varphi_2(z)}\|_q, \forall z \in \CC.
\end{align*}
Thus, for every $z \in \CC$,
\begin{align*}
 2\|L-T\|  \geq & \left( m_z(\psi_1,\varphi_1) +  m_z(\psi_2,\varphi_2) \right) \left( 1 - e^{-\frac{|\varphi_1(z)-\varphi_2(z)|^2}{2}} \right) \\
& \qquad \qquad \qquad \qquad \qquad \quad  - \|Tk_{\varphi_1(z)}\|_q - \|Tk_{\varphi_2(z)}\|_q \\
 \geq & 2 \rho(\varphi_1(z),\varphi_2(z)) \left( m_z(\psi_1,\varphi_1) +  m_z(\psi_2,\varphi_2) \right) \\
& \qquad \qquad \qquad \qquad \qquad \quad  - \|Tk_{\varphi_1(z)}\|_q - \|Tk_{\varphi_2(z)}\|_q.
\end{align*}
Since $a_1 \neq 0$ and $a_2 \neq 0$, $\varphi_1(z) \to \infty$ and $\varphi_2(z) \to \infty$ as $|z| \to \infty$. Then, by Lemma \ref{lem-com}, $Tk_{\varphi_1(z)} \to 0$ and $Tk_{\varphi_2(z)} \to 0$ in $\mathcal F^q$ as $|z| \to \infty$.

Using this and letting $|z| \to \infty$, we get
\begin{align*}
\|L-T\| \geq \limsup_{|z| \to \infty} \left( m_z(\psi_1,\varphi_1) +  m_z(\psi_2,\varphi_2) \right) \rho(\varphi_1(z),\varphi_2(z)).
\end{align*}
 From this the lower estimate for $\|L\|_e$ follows. 
\begin{flushright}
$\square$
\end{flushright}

\begin{rem}
The method in this paper works for differences (and hence for sums) of only two weighted composition operators on Fock spaces and it cannot apply to more than two such operators. 

A natural question to ask is: how's about a sum of $n \geq 3$ weighted composition operators on Fock spaces? To solve this question a new approach is supposed to develop. So far as we know, there is not yet any correct solution to this case, even for linear combinations of composition operators on Fock spaces.
\end{rem}

\bigskip
 
\textbf{Acknowledgement.} The article was completed during the first-named author's stay at Division of Mathematical Sciences, School of Physical and Mathematical Sciences, Nanyang Technological University, as a postdoc under the MOE's AcRF Tier 1 M4011724.110 (RG128/16). He would like to thank the institution for hospitality and support.

\end{document}